\documentclass{article}

\usepackage{bbm}
\usepackage[utf8]{inputenc}
\usepackage[english]{babel}
\usepackage{color}
\usepackage{amsmath,amssymb,amsthm,yhmath,amsfonts,mathrsfs,enumitem,mathtools}
\usepackage{tikz}
\usetikzlibrary{shapes,arrows,positioning}

\newcommand{\pr}{\mathbb{P}}
\newcommand{\E}{\mathbb{E}}
\renewcommand{\d}{\mathrm{d}}
\newcommand{\erg}{Erd\H os--R\'enyi graph}
\newcommand\Lleq{\stackrel{\mathcal{L}^N_M}{ \leq}}

\newtheorem{theorem}{Theorem}
\newtheorem{lemma}{Lemma}

\DeclareMathOperator{\cov}{cov}

\title{Mean field limit of local density-dependent Markov processes on weighted \erg s
}

\author{
	D\'aniel Keliger\\
	{\small Department of Stochastics,}\\
	{\small Budapest University of Technology and Economics}\\
	{\small e-mail: perfectumfluidum@gmail.com}\\[3mm]
	Ill\'es Horv\'ath\\
	{\small MTA-BME Information Systems Research Group}\\
	{\small e-mail: horvath.illes.antal@gmail.com}
}

\date{}

\begin{document}
	
%\sloppy
	
	\maketitle
	
\begin{abstract}
We study the asymptotic behaviour of Markov processes on large weighted \erg s  where the transition rates of the vertices are only influenced by the state of their neighbours and the corresponding weight on the edges. We find the ratio of vertices being in a certain state will converge to the solution of a differential equation obtained from mean field approximation if the graph is dense enough, namely, the average degree is at least of order $N^{\frac{1}{2}+\epsilon}$. Proof for convergence in probability in the transient regime is shown. %In the special case when the differential equation has only on fix point which is globally stable, this result could be extended to the asymptotic regime too.

\end{abstract}

\section{Introduction}
\label{s:introduction}

In the 1970's there was an interest in examining density dependent Markov processes. These stochastic processes have the property that the transition rates of individual components depend on the empirical distribution of other components (the density within the entire population). This way, the individual components interact only through the global average.

The area was pioneered by Kurtz, who proved that as the number of components increases, in the transient regime the process converges to the solution of a system of differential equations \cite{kurtz70}. He also proved that the error is $O\left( \frac{1}{\sqrt{N}} \right)$ where $N$ is the number of components and also proved a version of the central limit theorem \cite{kurtz78}.

These processes are widely popular till this day as they provide a wide range of applications, including  chemical reactions and queuing systems \cite{kurtz2011, telek2008}. However, their applicability is limited by the assumption that components interact with the global average.

Since then, investigations have been extended to deterministic and stochastic processes on large networks where the vertices interact only through their neighbours. One large class of such processes is the \emph{local density dependent Markov population processes}, the natural local version of the density dependent Markov process, where the evolution of each component is governed by the empirical distribution of its neighbours instead of the entire population.

The underlying graph structure turns out to be very important for the mean-field limit. The structure of real-life networks is often modelled by random graphs such as the \erg, the configuration model, random regular graphs etc. Depending on the structure of the network and the process, the evolution of the population process may not have an explicit closed form, so approximation is required.

The limit process provided by Kurtz is commonly used as a rough approximation, referred to as the \emph{homogeneous mean-field approximation} in this paper. This approximation ignores the structure completely, operating under the assumption that the distribution of the neighbourhood of each vertex is close to the global distribution. The validity of this assumption largely depends on whether the population process is sufficiently mixed. It has been known to fail in certain models, notably the SIR process in sparse graphs \cite{volz,volzproof}, where the mean-field limit is different from the homogeneous mean-field approximation.

In the present paper, we address the case of the \erg, and prove transient convergence to the homogeneous mean-field limit of Kurtz provided that the graph is sufficiently dense.

The rest of the paper is structured as the following: Section \ref{s:related} details related work, including other approximations and results where either mean-field convergence or bounds are shown in a rigorous way. Section \ref{s:setup} contains the mathematical framework for the model of the present paper along with important notation. The main result with proof is presented in Section \ref{s:simple} for a simple setup and \ref{s:general} for a general setup, and Section \ref{s:concl} concludes the work.

\section{Related works}
\label{s:related}

Kurtz proved several general results \cite{kurtz70,Kurtz1978} for density dependent Markov-processes as stated in Section \ref{s:introduction}.

A different approximation called \emph{heterogeneous mean-field approximation} is introduced in \cite{vesp}, taking into account the degree distribution and degree correlation of the network.

A lot of attention went to studying diffusion phenomena like disease spreading (e.g. the SIR model \cite{sir}) and opinion dynamics \cite{vesp, opinion}. For the SIR process on sparse graphs generated by the configuration model, the approach of Volz turned out to be exact \cite{volz, volzproof}.

Another approach is called the quenched mean field approximation or N-intertwined mean field approximation (NIMFA). This approach takes the full network topology into account, and approximates the evolution of each node by its expectation. Unfortunately the resulting system of differential equations is not closed, so standard differential equation techniques cannot be used. The standard approach \cite{nodemeanfield} is to assume the dynamical correlation is small. Overall, NIMFA requires $O(N)$ differential equations to be solved numerically for a network of $N$ vertices, which means that for the transient regime it can only be used for relatively small graphs.

Despite using the full topology of the network, few rigorous results are known about when NIMFA performs well. One important example is when dynamical correlations are nonnegative, in which case NIMFA can be used to make pessimistic predictions about prevalence of diseases on moderate sized networks \cite{simon2017NIMFA}.

Finally, a recent result by Xue \cite{vertexweight} showed that the SIR process on dense \erg s with bounded vertex weights can be described by a small set of differential equations.

\section{Setup of the model}
\label{s:setup}

We first set up the local density-dependent Markov process, then the homogeneous mean-field approximation, and finally the NIMFA approximation.

\subsubsection*{Local density-dependent Markov process}

$G^N$ will denote an \erg{} with $N$ vertices (numbered from $1$ to $N$) with random weights $a_{ij}^N$ which describe the strength of the connection between vertices $i$ and $j$.

We assume $G^N$ is undirected and has no loops, which means, we have $a_{ii}^N=0 $ and $a_{ij}^N=a_{ji}^N$.

We decompose $a_{ij}^N$ into the product
$$a_{ij}^N=w_{ij}^Nb_{ij}^N,$$
where the $w_{ij}^N$'s and $b_{ij}^N$'s are independent; $b_{ij}^N$ is the indicator of whether a link is present between $i$ and $j$, and $w_{ij}^N$ stands for the weight of that link. We assume $w_{ij}^N$ are i.i.d. for $i<j$ with a common nonnegative real distribution independent from $N$, and $b_{ij}^N$ are also i.i.d. for $i<j$, with $\pr(b_{ij}^N=1)=p^N$ allowing for scaling with $N$.

The unweighted {\erg} is obtained by setting $w_{ij}^N=1$.

We assume that the distribution of $w_{ij}^N$ has a finite mean $\mu$ and variance $\sigma^2$.

By the law of total expectation it is easy to see that
\begin{align*}
&\tilde{\mu}:=\E\left( a_{ij}^N\right)=p^N \mu \\
& \tilde{\sigma}^2:=\mathbb{D}^2 \left( a_{ij}^N\right)=p^N\left(\sigma^2+\mu^2\right)-\left(p^N \right)^2 \mu^2.
\end{align*}

The sum of the weights originating from vertex $i$ is denoted by
$$d^N(i)=\sum_{j=1}^Na_{ij}^N.$$
(In the unweighted case $d^N(i)$ corresponds to the degree of vertex $i$.)

Clearly
\begin{align*}
\E \left(d^N(i)\right)=(N-1)\tilde{\mu}=:\langle d \rangle. 
\end{align*}

$G^N$ serves as a random environment for a local density-dependent Markov process as described below. 

Each vertex can be in one of the states from the local state space $\mathcal{S}=\{1,...,S\}$. The \emph{global state} of the system is defined as the total number of vertices in each state, that is, a vector $X^N\in \{0,1,\dots,N\}^{S}$ with $X^N_1+\dots+X^N_{S}=N$. The normalized global state, which records the ratio of vertices in each state, is simply $x^N=\frac{X^N}{N}$.

In the classic density-dependent mean field setup of Kurtz \cite{kurtz70,kurtz78}, each vertex changes its local state according to a continuous time Markov process with rates depending on $x^N$. If $\xi_{s,i}^N(t)$ denotes the indicator that vertex $i$ is in state $s$ at time $t$, then the normalized global state vector $x^N(t)=(x_s^N(t))_{s\in\mathcal{S}}$ at time $t$ can be written as
\begin{align}
\label{eq:xdef}
x_s^N=x_s^N(t)=\frac1N \sum_{i=1}^N\xi_{s,i}^N(t)
\end{align}
and the rate of vertex $i$ transitioning from state $s$ to $k$ is $Q_{ks}=Q_{ks}\left( x^N\right)$. We will also use the notation $\xi_{i}^N(t)= \left(\xi_{s,i}^N(t)\right)_{s\in\mathcal{S}}$.

We note that for a general $G^N$ graph, $x^N(t)$ is not a Markov process. The entire collection of $\xi_{i}^N(t)$'s (referred to as the \emph{exact system}) is a Markov process, but on a state space exponentially large in $N$.

In the local density-dependent version of this process, we look to set the transition rates of each vertex depend on the weighted average of its neighbours with weights $a_{ij}^N$ instead of the simple average of the entire system of Kurtz. In this setting, vertices connected with an edge with larger weights have a stronger influence on each other. Also, the weighted averages could be different for different vertices, hence, they have a different ``perception'' of the global state. In the unweighted case, this means that the transition rates of a certain vertex depend only on the state of its neighbours. This motivates the definition of the vector
\begin{align*}
\phi_{i}^N(t):=\frac{1}{\langle d \rangle}\sum_{j=1}^Na_{ij}^N \xi_{j}^N(t).
\end{align*}
For technical simplicity, it is normalized by the deterministic global value $\langle d \rangle$ so the transitions rates will not blow up as we increase $N$.

For the \erg{} model, the rate of vertex $i$ transitioning from state $s$ to $k$ is
$$Q_{ks}=Q_{ks}\left( \phi_{i}^N(t)\right).$$

We use the convention
$$Q_{ss}\left(\phi_{i}^N(t)\right):=-\sum_{k \neq s}Q_{ks}\left(\phi_{i}^N(t)\right)$$
for diagonal elements. These rate functions can also be written in the more compact matrix form
$$Q\left( \phi \right)=\left( Q_{ks}\left(\phi\right) \right)_{k,s=1}^{S}.$$

Throughout the paper, we will assume that all $Q_{sk}$'s are locally Lipschitz-continuous on $\mathbb{R}^\mathcal{S}$. On $\mathbb{R}^\mathcal{S}$, we use $\|.\|=\|.\|_1$ (but all norms on $\mathcal{R}^\mathbb{S}$ are equivalent anyway). $\|.\|$ will also use the corresponding operator norm on $\mathbb{R}^\mathcal{S}\to\mathbb{R}^\mathcal{S}$ functions.

\subsubsection*{Homogeneous mean-field approximation}

The homogeneous mean-field approximation of $x^N(t)$ is $u(t)=\{u_s(t)\}_{s\in\mathcal{S}}$, the solution of
\begin{align*}
\frac{\d}{\d t}u_s(t)=\sum_{k\in\mathcal{S}}Q_{sk}(u(t))u_k(t),
\end{align*}
written in a more compact form as
\begin{align}
\label{eq:mf}
\frac{\d}{\d t}u(t)=\underbrace{Q(u(t))u(t)}_{f(u(t))}.
\end{align}

%\eqref{eq:mf} essentially approximates the neighbourhood of each vertex by the global state. %Notice that \eqref{eq:mf} only has $S$ components.

Since $Q$ is locally Lipschitz-continuous, it is also Lipschitz continuous on the simplex
$$\Delta:=\{ v \in \mathbb{R}^S : v_{s} \geq 0\,\forall s \in \mathcal{S}, \ \sum_{s=1}^S v_{s}=1 \}$$
along with $f$. Existence and uniqueness of a global solution of \eqref{eq:mf} on $\Delta$ follows.

The Lipschitz-constant for $Q_{sk}$ is denoted by $L_{sk}$ and the Lipschitz-constant of $f$ is denoted by $L_f$ on the appropriate domain (which will be defined later); $L_Q$ denotes the Lipschitz-constant of the operator norm of $Q$ with respect to the operator norm, again on the appropriate domain.

\subsubsection*{NIMFA}

$\E_G$ and $\pr_{G}$ denote conditional expectation and probability conditioned on $G^N$, that is, we only average out the randomness of the stochastic process, not the environment.
 
The probabilities $\pr_{G}\left(\xi_{s,i}^N(t)=1 \right)=\E_{G} \left(\xi_{s,i}^N(t)\right)$ satisfy the differential equation (in vector form)
\begin{align}
\label{eq:NIMFA3}
&\frac{d}{dt}\E_{G}\left(\xi_{i}^N(t)\right)=\E_{G}\left[ Q\left( \phi_{i}^N(t) \right) \xi_{i}^N(t)  \right]
\end{align}

Unfortunately \eqref{eq:NIMFA3} is not closed, so ODE techniques can not be used directly to study the behaviour of the probabilities. To get a closed system, we introduce the N-intertwined mean field approximation (NIMFA) in a more general form. NIMFA approximates the environment $\phi_i(t)$ of each vertex $i$ by its average $\E \left(\phi_i^N(t) \right) $:
\begin{align}
\label{eq:mfappr}
\E_{G}\left( Q\left( \phi_{i}^N(t) \right) \xi_{i}^N(t)  \right) \approx Q\left( \E_{G}\left(\phi_{i}^N(t) \right)\right)\E_{G}\left(\xi_{i}^N(t)\right),
\end{align}
which results in the closed dynamic
\begin{align}
\label{eq:deterministic}
\begin{split}
& \dot{z}_{i}^N(t)=Q\left( \rho_{i}^N(t)\right)z_{i}^N(t),  \\
& \rho_{i}^N(t)=\frac{1}{\langle d \rangle}\sum_{j=1}^Na_{ij}^Nz_{j}^N(t).
\end{split}
\end{align}

Despite \eqref{eq:deterministic} being a closed system, it contains $SN$ equations, which is difficult to handle for large $N$ (yet still much better than the exact system).

Local Lipschitz continuity guarantees that local solutions of \eqref{eq:deterministic} exist and are unique. For initial conditions starting from $\Delta$, we can extend this result to global solutions. Global solutions stay in $\Delta$, giving a nice probabilistic meaning to them (see Lemma \ref{l:appendix} in the Appendix).

The NIMFA approximation for $x^N(t)$ is then
\begin{align}
\label{eq:avg}
&y^N(t):=\frac{1}{N}\sum_{i=1}^Nz_{i}^N(t).
\end{align}

In Sections \ref{s:simple} and \ref{s:general}, we present sufficient conditions that guarantee that $x^N(t), u(t)$ and $y^N(t)$ are close in the transient regime. Section \ref{s:simple} focuses on the simple dense {\erg}, presenting the main ideas of the proof in a simple setting, while Section \ref{s:general} examines the general setting.

\section{Mean-field for simple \erg}
\label{s:simple}

First we consider the case when $G^N$ is the simple \erg, that is, the random weights $a_{ij}^N$ are $1$ with probability $p$ and $0$ with probability $1-p$.  $ a_{ij}^N=a_{ji}^N$ (the graph is undirected), $a_{ii}^N=0$ (no loops), and for $i<j$, $a_{ij}^N$ are i.i.d. The regime we consider is the dense \erg, e.g. when $p$ is constant and $N\to\infty$.

Due to
\begin{align}
\label{eq:1perp}
\|\phi_i^N(t)\|= \frac{d^N(i)}{\langle d \rangle} \leq \frac{N-1}{\langle d \rangle}=\frac{1}{p},
\end{align}
$Q_{sk}$ can be restricted to the compact domain $\{\|\phi\|\leq 1\}$; its Lipschitz-constant on this domain is denoted by $L_{sk}$, and $L_Q$ is the operator norm of $Q$.

Also assume $x^N(0)\xrightarrow{p} u(0) $ and $y^N(0)\xrightarrow{p} u(0) $ respectively as $N\to\infty$ ($\xrightarrow{p}$ denotes convergence in probability).

\begin{theorem}
	\label{t:simple} In the regime described above, for every fixed $T>0$,
	\begin{align}
	\label{eq:xu0}
	&\sup_{0 \leq t \leq T}\left\| x^N(t)-u(t) \right\| \xrightarrow{p}0
	\end{align}
	and
	\begin{align}
	\label{eq:yu0}
	&\sup_{0 \leq t \leq T}\left\| y^N(t)-u(t) \right\| \xrightarrow{p}0
	\end{align}
as $N\to\infty$.
\end{theorem}

For the proof of \eqref{eq:xu0}, we use Poisson representation similar to modern proofs of Kurtz (see e.g. \cite{Hayden2012}):
\begin{align}
\nonumber
x_s^N(t) = x_s^N(0) + \frac1N\sum_{k\neq s}Y_{sk}\left(
\int_{0}^t \sum_{i=1}^N Q_{sk}(\phi_i^N(\tau))\xi_{i,k}^N(\tau)\d \tau\right)\\
\label{eq:poirepr}
-\frac1N\sum_{k\neq s}Y_{ks}\left(
\int_{0}^t \sum_{i=1}^N Q_{ks}(\phi_i^N(\tau))\xi_{i,s}^N(\tau)\d \tau\right),
\end{align}
where $Y_{ks}, s,k\in\mathcal{S}$ are independent Poisson processes with rate 1. Technically, the Poisson-processes can be chosen to be the same for any $N$, which presents a coupling between systems for various values of $N$, but this coupling will not be relevant.

We introduce the error terms due to averaging out fluctuations of the Poisson-processes.
\begin{align}
\nonumber
K_{sk}^N(t)=&\frac1N Y_{sk}\left(
\int_{0}^t \sum_{i=1}^N Q_{sk}(\phi_i^N(\tau))\xi_{i,k}^N(\tau)\d \tau\right)\\
\label{eq:kdef}
&-\int_{0}^t\frac1N \sum_{i=1}^N Q_{sk}(\phi_i^N(\tau))\xi_{i,k}^N(\tau)\d \tau
\end{align}
and
\begin{align}
\label{eq:kdef}
K^N(t)&=\Big(\sum_{k:k\neq s}K_{sk}^N(t)-K_{ks}^N(t)\Big)_{s\in\mathcal{S}}
\end{align}
and also the error terms due to the difference in local and global environments:
\begin{align}
\label{eq:hdef}
H^N(t)=\frac1N \sum_{i=1}^N Q(\phi_i^N(t))\xi_{i}^N(t) - Q(x^N(t))x^N(t)
\end{align}
so
\begin{align}
\nonumber
x^N(t) = x^N(0) + K^N(t) + \int_0^t H^N(\tau)\d\tau+\int_0^t f(x^N(\tau))\d\tau.
\end{align}

To handle the error terms, we need certain properties of the underlying graph structure. To simplify notation, we introduce the random variable $\eta$ which takes values on $\{1,...,N\}$ uniformly independent from all other variables. $\eta$ mostly serves as a tool to rewrite large summations into shorter formulas. Accordingly, we also use the notation $\E_{G,\xi}$ which refers to the conditional expectation with respect to the random environment $G^N$ and the stochastic process $\xi_i^N(t)$ on it, only averaging out the artificial randomness introduced by $\eta$.

We introduce the terms
\begin{align*}
& c^N(i,j):=\cov_{G}\left( a_{i \eta}^N, a_{j \eta}^N \right)= \frac{1}{N}\sum_{k=1}^Na_{ik}^Na_{jk}^N -\frac{1}{N}\sum_{k=1}^Na_{ik}^N\frac{1}{N}\sum_{l=1}^Na_{jl}^N.
\end{align*} 
To make this formula simpler, we will use the notation
\begin{align*}
&d^N(i,j):=\sum_{k=1}^Na_{ik}^Na_{jk}^N
\end{align*}
for the number of common neighbours of vertices $i$ an $j$. Then 
\begin{align*}
&c^N(i,j)=\frac{1}{N}d^N(i,j)-\frac{1}{N^2}d^N(i)d^N(j).
\end{align*}
For $i<j$, the $c^N(i,j)$'s are identically distributed. The diagonal terms have the trivial bound
\begin{align*}
& 0 \leq c^N(i,i) \leq \frac{1}{N}\sum_{k=1}^N\left(a_{ik}^N\right)^2 = \frac{d^N(i)}{N}.
\end{align*}

\begin{lemma}
	\label{l:simplec}
	\begin{align}
	\label{eq:cm1}
	\E_G\left(\left|c^N(i,j) \right|\right) \to 0
	\end{align}
	as $N \to \infty$. 
\end{lemma}
\noindent Proof. Since the $c^N(i,j)$'s are identically distributed, it is sufficient to examine $c^N(1,2)$.
\begin{align*}
c^N(1,2)&= \frac{1}{N}\sum_{k=1}^Na^N_{1k}a^N_{2k} -\frac{1}{N}\sum_{k=1}^Na^N_{1k}\frac{1}{N}\sum_{l=1}^Na^N_{2l}.
\end{align*}
For each $k$, $a_{1k}^N$ and $a_{2k}^N$ are independent Bernoulli variables with parameter $p$, and the $a_{1k}^Na_{2k}^N$'s are independent Bernoulli variables with parameter $p^2$, except for $k=1,2$ when one of them is $0$. Therefore, by the law of large numbers $\frac{1}{N}\sum_{k=1}^Na^N_{1k}a^N_{2k} \xrightarrow{p} p^2$.  Similarly, $\frac{1}{N}\sum_{k=1}^Na^N_{ik} \xrightarrow{p} p  $ for $i=1,2$, hence $c^N(1,2) \xrightarrow{p} p^2-p \cdot p=0$.

The proof is concluded by the observation that $|c^N(1,2)| \leq 1$, so $c^N(1,2)$ also converges in moments. $\Box$
\begin{lemma}
	\label{l:hnt1}
	There exists a sequence of random variables $R_1^N$ not depending on $t$ such that $R^N$ is $G^N$-measurable,
	\begin{align*}
	\left\|\frac1N \sum_{i=1}^N \phi_i^N(t)-x^N(t)\right\|\leq R_1^N
	\end{align*}
	and
	\begin{align*}
	R_1^N \xrightarrow{p}0
	\end{align*}
	as $N\to\infty$.
\end{lemma}

Lemma \ref{l:hnt1} states that the averages of local environments are close to the global environment for large $N$.

\noindent Proof.

\begin{align*}
&  \frac{1}{N}\sum_{i=1}^{N}\phi_{i}^N(t)=\frac{1}{N\langle d \rangle}\sum_{i=1}^{N}\sum_{j=1}^N a_{ij}^N \xi_{j}^N(t)= \\
&\frac{1}{N\langle d \rangle}\sum_{j=1}^Nd^N(j)\xi_{j}^N(t)=x^N(t)+\frac{1}{N\langle d \rangle}\sum_{j=1}^N \left[d^N(j)-\langle d \rangle \right]\xi_{j}^N(t),
\end{align*}
 and so
\begin{align}
\label{l2:useful:1}
\begin{split}
& \left \| \frac{1}{N}\sum_{i=1}^{N}\phi_{i}^N(t)-x^N(t) \right \| \leq \frac{1}{N\langle d \rangle}\sum_{j=1}^N \left|d^N(j)-\langle d \rangle \right| \ \cdot \left \|\xi_{j}^N(t) \right \| = \\
&\frac{1}{N\langle d \rangle}\sum_{j=1}^N \left|d^N(j)-\langle d \rangle \right| \leq \sqrt{\frac{1}{N \langle d \rangle^2}\sum_{j=1}^N\left(d^N(j)-\langle d \rangle \right)^2}=:R_1^N.
\end{split}
\end{align}
$R_1^N$ is clearly $G^N$-measurable, and
\begin{align}
\label{l2:useful:2}
\left(\E(R_1^N)\right)^2\leq \E\left((R_1^N)^2\right)=&
\frac{1}{N \langle d \rangle^2} \sum_{j=1}^N \mathbb{D}^2 \left( d^N(j)\right)=\frac{\mathbb{D}^2\left( d^N(1)\right)}{\langle d \rangle^2}=\\
\nonumber
&\frac{(N-1)p(1-p)}{((N-1)p)^2}=O\left( \frac{1}{N} \right) \to 0
\end{align}
as $N\to\infty$. So $\E(R_1^N)\to 0$ as $N\to\infty$, and Markov's inequality finishes the proof of Lemma \ref{l:hnt1}. $\Box$

Lemma \ref{l:Dfismall} states that local environments do not vary much when $N$ is large. Combining this with lemma \ref{l:hnt1}, we obtain that for large $N$ nodes typically observe a similar environment which is close to the global state.

\begin{lemma}
\label{l:Dfismall}
	There exists a sequence of random variables $R^N_2$ not depending on $t$ such that $R^N_2$ is $G^N$-measurable,
	\begin{align*}
	& \sum_{s=1}^S \mathbb{D}^2_{G, \xi}\left( \phi_{\eta,s}^N(t) \right) \leq R_2^N
	\end{align*}
	and
	\begin{align*}
	R^N_2 \xrightarrow{p}0
	\end{align*}
	as $N\to\infty$.

\end{lemma}

\noindent Proof.
\begin{align}
\label{l3:useful:1}
\begin{split}
& \sum_{s=1}^S \mathbb{D}^2_{G, \xi}\left( \phi_{\eta,s}^N(t) \right)= \sum_{s=1}^S \mathbb{D}^2_{G, \xi}\left( \frac{1}{\langle d \rangle}\sum_{i=1}^N a_{\eta i}^N \xi_{i,s}^N(t) \right)= \\
& \frac{1}{\langle d \rangle^2}\sum_{i=1}^N \sum_{j=1}^N \underbrace{\sum_{s=1}^N \xi_{i,s}^N(t) \xi_{j,s}^N(t)}_{0\leq .\leq 1} \cov_G\left(a_{\eta i}^N, a_{\eta j}^N \right)\leq\frac{1}{\langle d \rangle^2}\sum_{i=1}^N \sum_{j=1}^N |c^N(i,j)|=:R_2^N.
\end{split}
\end{align}

$R_2^N$ is clearly $G^N$-measurable. To prove $R_2^N\to 0$ in probability, we handle the diagonal and the non-diagonal terms separately.

Since $0 \leq c^N(i,i) \leq \frac{1}{N}d^N(i) $, we have
\begin{align*}
& \E \left [\frac{1}{\langle d \rangle^2}\sum_{i=1}^N \left |c^N(i,i) \right | \right]=\frac{N}{\langle d \rangle^2} \E \left |c^N(1,1) \right | \leq  \frac{1}{\langle d \rangle^2} \E \left (d^N(1)\right)= \\
&\frac{1}{\langle d \rangle}=\frac{1}{(N-1)p} \to 0.
\end{align*}

For the non-diagonal terms, we use Lemma \ref{l:simplec}:
\begin{align*}
&\E \left[\frac{2}{\langle d \rangle^2}\sum_{i<j}  \left |c^N(i,j) \right | \right]=\frac{N(N-1)}{\langle d \rangle^2} \E \left( \left|c^N(1,2) \right | \right)= \\
&\frac{N}{(N-1)p^2}\E \left( \left|c^N(1,2) \right | \right) \to 0.
\end{align*}
so $ R_2^N \xrightarrow{p}0$ holds. $\Box$

Now we turn our attention specifically to the error terms $K^N(t)$ and $H^N(t)$. 

\begin{lemma}
	\label{l:knt}
	For any $T>0$ fixed,
	\begin{align}
	\sup_{0\leq t\leq T}\|K^N(t)\| \xrightarrow{p}0
	\end{align}
	as $N\to\infty$.
\end{lemma}

\noindent Proof (Lemma \ref{l:knt}). It is sufficient to show $\sup_{0\leq t\leq T}|K_{sk}^N(t)| \xrightarrow{p}0$ for any $s\neq k\in \mathcal{S}$.
Then using \eqref{eq:1perp},
\begin{align*}
Q_{sk}(\phi_i^N(t))\xi_{i,k}^N(t)\leq L_{sk}\|\phi_i^N(t)\|&\leq \frac{L_{sk}}{p},\\
\int_{0}^t\frac{1}{N}\sum_{i=1}^N Q_{sk}(\phi_i^N(\tau))\xi_{i,k}^N(\tau)\d \tau&\leq  \frac{L_{sk} t}{p}
\end{align*}
and
\begin{align}
\nonumber
\sup_{0\leq t\leq T} |K_{sk}^N(t)|\leq &\sup_{0\leq x\leq \frac{L_{sk}T}{p}}\left| \frac1N Y_{sk}(Nx)-x\right|,
\end{align}
which goes to 0 in probability according to the functional strong law of large numbers for the Poisson process (\cite{wardsuppl}, Section 3.2). $\Box$

\begin{lemma}
	\label{l:hnt2}
	\begin{align*}
	\sup_{0\leq t\leq T} \|H^N(t)\| \xrightarrow{p}0
	\end{align*}
	as $N\to\infty$ for all $T>0$.
\end{lemma}

\noindent Proof.

\begin{align}
\label{l5:useful:1}
\begin{split}
 \left \| H^N(t) \right\| &= \left \| \frac{1}{N}\sum_{i=1}^NQ \left(\phi_i^N(t) \right)\xi_{i}^N(t)-f \left( x^N(t) \right) \right\| \\
&= \left \|\E_{G, \xi} \left[Q\left(\phi_{\eta}^N(t)\right) \xi_{\eta}^N(t)\right]-Q\left(x^N(t) \right)x^N(t) \right \| \\
&= \left \| \E_{G, \xi} \left[ \left(Q \left( \phi_{\eta}^N(t)\right)-Q \left( x^N(t) \right) \right)\xi_{\eta}^N(t) \right] \right \| \\
& \leq  \E_{G, \xi} \left \| \left(Q \left( \phi_{\eta}^N(t)\right)-Q \left( x^N(t) \right) \right)\xi_{\eta}^N(t)  \right \| \\
& \leq \E_{G, \xi} \left ( \left \| Q \left( \phi_{\eta}^N(t)\right)-Q \left( x^N(t) \right)  \right \| \cdot \left \|\xi_{\eta}^N(t)  \right \| \right )\\
& =\E_{G, \xi} \left \| Q \left( \phi_{\eta}^N(t)\right)-Q \left( x^N(t) \right)  \right \|\\
&\leq L_Q \E_{G, \xi} \left \|\phi_{\eta}^N(t)-x^N(t) \right \| \\
&\leq L_Q \left[ \left \| \E_{G,\xi}\left(\phi_{\eta}^N(t) \right)-x^N(t) \right \|+\E_{G,\xi}\left \| \phi_{\eta}^N(t) -\E_{G,\xi}\left(\phi_{\eta}^N(t) \right)\right \| \right] \\
& \leq L_Q \left( R_1^N+\sum_{s=1}^S \E_{G,\xi}\left | \phi_{\eta,s}^N(t) -\E_{G,\xi}\left(\phi_{\eta,s}^N(t) \right) \right | \right)\\
& \leq L_Q \left( R_1^N+ \sqrt{S \sum_{s=1}^S \mathbb{D}^2_{G,\xi} \left( \phi_{\eta, s}^N(t)\right)}  \right) \leq L_Q\left(R_1^N+\sqrt{SR_2^N} \right) \xrightarrow{p} 0
\end{split}
\end{align}
as $N\to\infty$.  $\Box$

To finish the proof of Theorem \ref{t:simple}, we have
\begin{align}
\nonumber
x^N(t)-u(t) = & x^N(0)-u(0) + K^N(t) + \int_0^t H^N(\tau)\d\tau\\
&+ \int_0^t f(x^N(\tau)) - f(u^N(\tau))\d\tau,
\end{align}
and Gr\"onwall gives
\begin{align}
\nonumber
&\sup_{0\leq t\leq T} \|x^N(t)-u(t)\|\leq\\
\nonumber
&\Big(\|x^N(0)-u(0)\| + \underbrace{\sup_{0\leq t\leq T} \|K^N(t)\| + T\sup_{0\leq t\leq T} \|H^N(t)\|}_{\xrightarrow{p}0}\Big)e^{L_f T},
\end{align}
proving \eqref{eq:xu0}.

The proof of \eqref{eq:yu0} is essentially identical to the proof of \eqref{eq:xu0} as all bounds come inherently from the graph structure. To prove \eqref{eq:yu0}, one has to use that the solutions can be interpreted as probability vectors, then use the same trivial upper bounds as for the indicators. Note that in the case of NIMFA the error term $K^N(t)$ is absent. Details are left to the reader. $\Box$

\section{Mean-field for general \erg}
\label{s:general}

In this section, we explore how far the ideas of the proof in Section \ref{s:simple} can be taken to extend the result to general \erg{}s. We use the notation of Section \ref{s:setup}.

Our baseline assumptions in this section are the following:
\begin{itemize}
	\item convergence of the initial conditions, that is, $x^N(0),y^N(0) \xrightarrow{p}u(0)$ as $N\to\infty$, and
	\item the rate functions $Q_{s,k}$ are locally Lipschitz-continuous for all $s,k \in \mathcal{S}$.	
\end{itemize}

We list some additional assumptions we will need.

\begin{itemize}
	\item[] A1: For all $s,k \in \mathcal{S}$, $Q_{s,k}$ is globally Lipschitz continuous with Lipschitz-constant $L_{sk}$.
	\item[] A2: There are positive constants $c, \epsilon$ such that $\langle d \rangle \geq c N^{\frac{1}{2}+\epsilon}$.
\end{itemize}

In many applications in epidemiology, the rate functions are linear and A1 is not restrictive. However, if we only have locally Lipschitz continuous rate functions, we need additional regularity conditions for the weights.

\begin{itemize}
	\item[] B1: $p^N \equiv 1$
	\item[] B2: the logarithmic generator function $\Lambda(s)=\log(\E(\exp(s w_{12}^N)))$ is finite for some $s>0$. 
\end{itemize}

A fixed constant factor in the weights could be incorporated into either B1 or B2 (in the current formulation, it is incorporated into B2).

We note that B1 implies A2.

%\begin{itemize}
%	\item[] C1: \eqref{eq:mf} has a unique fix point $u(\infty)$ in $\Delta$ which is globally stable in the simplex. 
%\end{itemize}

%C1 is needed to ensure that $u(t)$ has the same asymptotic behaviour as $x^N(t)$ and $y^N(t)$ for large $N$.

\begin{theorem}
\label{t:1} Assume the baseline assumptions plus either (A1 and A2) or (B1 and B2). Then for every $T>0$,
\begin{align}
\label{eq:xu}
&\sup_{0 \leq t \leq T}\left\| x^N(t)-u(t) \right\| \xrightarrow{p}0
\end{align}
and
\begin{align}
\label{eq:yu}
&\sup_{0 \leq t \leq T}\left\| y^N(t)-u(t) \right\| \xrightarrow{p}0.
\end{align}
\end{theorem}

%The exact choice of norm is irrelevant due to the equivalence of norms on $\mathbb{R}^S$. $\|.\|$ will be used during the proof.

%For the asymptotic behaviours we have the following two theorems:

%\begin{theorem}
%\label{t:2} Assume either A1,A2 or B1,B2 and C1. If the limit $z_i^N(t) \to z_i^N(\infty)$ exist for all $N$ and $i$, than
%\begin{align}
%& y^N(\infty) \xrightarrow{p}u(\infty).
%\end{align}	
%\end{theorem} 

%\begin{theorem}
%\label{t:3} Assume either A1,A2 or B1,B2 and C1. 
%EZT MÉG JÓL ÁT KELL GONDOLNI 
%\end{theorem}

Following the notation of Section \ref{s:simple}, we again use $\eta$, a random variable that takes values on $\{1,...,N\}$ uniformly, independently from all other variables, along with the notation $\E_{G,\xi}$ which refers to the conditional expectation with respect to the random environment $G^N$ and the stochastic process $\xi_i^N(t)$.

We also use the same notation $K^N(t)$ and $H^N(t)$ as introduced in \eqref{eq:kdef} and \eqref{eq:hdef}.

\begin{align}
\nonumber
K_{sk}^N(t)=&\frac1N Y_{sk}\left(
\int_{0}^t \sum_{i=1}^N Q_{sk}(\phi_i^N(\tau))\xi_{i,k}^N(\tau)\d \tau\right)\\
\nonumber
&-\int_{0}^t\frac1N \sum_{i=1}^N Q_{sk}(\phi_i^N(\tau))\xi_{i,k}^N(\tau)\d \tau
\end{align}
with
\begin{align*}
K^N(t)&=\Big(\sum_{k:k\neq s}K_{sk}^N(t)-K_{ks}^N(t)\Big)_{s\in\mathcal{S}}
\end{align*}
and
\begin{align*}
H^N(t)=\frac1N \sum_{i=1}^N Q(\phi_i^N(t))\xi_{i}^N(t) - Q(x^N(t))x^N(t).
\end{align*}

Lemma \ref{l:1} is the counterpart of Lemma \ref{l:hnt1} for the general case.
\begin{lemma}
\label{l:1}
Assume A2. Then there is a random variable $R_1^N$ which is $G^N$-measurable, $R_1^N \xrightarrow{p} 0 $ and 
\begin{align}
&\left \| \frac{1}{N}\sum_{i=1}^{N}\phi_{i}^N(t)-x^N(t) \right \| \leq R_1^N .
\end{align}
\end{lemma}

Lemma \ref{l:1} can be reformulated with $\eta$ by observing that $\frac{1}{N}\sum_{i=1}^{N}\phi_{i}^N(t)=\E_{G, \xi}\left( \phi_{\eta}^N(t)\right)$.

\begin{proof}
$R_1^N$ is defined according to \eqref{l2:useful:1} again, and we have
\begin{align*}
&\left \| \frac{1}{N}\sum_{i=1}^{N}\phi_{i}^N(t)-x^N(t) \right \| \leq R_1^N, \\
& \left(E\left(R_1^N \right) \right)^2 \leq \frac{\mathbb{D}^2\left( d^N(1) \right)}{\langle d \rangle^2}.
\end{align*}
Then
\begin{align*}
&\frac{\mathbb{D}^2\left( d^N(1)\right)}{\langle d \rangle^2}=
\frac{1}{N-1}\frac{\tilde{\sigma}^2}{\tilde{\mu}^2}=\frac{1}{N-1} \frac{p^N\left(\sigma^2+\mu^2\right)-\left(p^N \right)^2 \mu^2}{\left(p^N\right)^2 \mu^2 } \leq \frac{1}{\langle d \rangle} \frac{\sigma^2+\mu^2}{\mu^2}\to 0
\end{align*} 
according to A2. So $\E\left(R_1^N \right) \to 0$ and $R_1^N \xrightarrow{p} 0$ as $N\to\infty$.
\end{proof}

\begin{lemma}
	\label{l:2}
\begin{align}
\mathbb{D}^2\left(N^2 c^N(1,2) \right)=O\left(N \langle d \rangle ^2 \right).
\end{align}
\end{lemma}

\begin{proof}
\begin{align}
\label{eq:cvar}
\begin{split}
&\mathbb{D}^2\left(N^2 c^N(1,2) \right)=\mathbb{D}^2\left(Nd^N(1,2)-d^N(1)d^N(2) \right)= \\
&N^2\mathbb{D}^2\left( d^N(1,2)\right)-2N \cov \left(d^N(1,2),d^N(1)d^N(2) \right)+\mathbb{D}^2\left(d^N(1)d^N(2) \right).
\end{split}
\end{align}
To address the various terms in \eqref{eq:cvar}, we use the notation
\begin{align*}
&\gamma^N(i,j,k,l):=\cov(a_{1i}^Na_{2j}^N, a_{1k}^Na_{2l}^N).
\end{align*}

The following symmetry properties hold:
\begin{align*}
&\gamma^N(i,j,k,l)=\gamma^{N}(j,i,k,l)=\gamma^N(i,j,l,k)=\gamma^N(k,l,i,j).
\end{align*}

We aim to identify and bound the nonzero terms among the $\gamma^N(i,j,k,l)$'s. 

$\gamma^N(i,j,k,l)$ can only differ from $0$ if either $i=k\neq 1$ or $j=l\neq 2$ due to independence of $a_{1i}^Na_{2j}^N$ and $a_{1k}^Na_{2l}^N$. Depending on whether one or both hold, there are essentially two distinct cases.

The first case: $i=k, j \neq l$.
\begin{align*}
&\gamma^N(i,j,i,k)=\E \left[ \left(a_{1i}^N\right)^2a_{2j}^Na_{2l}^N \right]-\E^2\left( a_{1i}^N\right)\E\left( a_{2j}^N\right)\E\left( a_{2k}^N\right)= \\
&p^N\left(\sigma^2+\mu^2\right)\left(p^N\mu \right)^2-\left(p^N\mu \right)^4=(p^N)^3\mu^2\left[\sigma^2+\left(1-p^N\right)\mu^2 \right]
\end{align*}
Overall, $0\leq \gamma^N(i,j,i,k)=O\left( \left(p^N\right)^3 \right)$ in this case.

The second case: $i=k, j=l$. Here, $\gamma^N(i,j,i,j)=\mathbb{D}^2 \left(a_{1i}^N a_{2j}^N \right) \geq 0$ and
\begin{align*}
&\gamma^N(i,j,i,j)=\E\left[\left(a_{1i}^N\right)^2\left(a_{2j}^N\right)^2 \right]-\E^2\left( a_{1i}^N\right)\E^2\left( a_{2j}^N\right)= \\
& \left[p^N(\sigma^2+\mu^2)\right]^2-\left(p^N \mu\right)^4=O\left( \left( p^N\right)^2 \right).
\end{align*}

Altogether, we have $\gamma^N(i,j,k,l)\geq 0$ for all $i,j,k,l$.

With the values of $\gamma^N$ above, estimating the terms in \eqref{eq:cvar} becomes a combinatorial problem:
\begin{align*}
& N^2 \mathbb{D}^2 \left( d^N(1,2) \right)=N^2\mathbb{D}^2 \left( \sum_{i=1}^N a_{1i}^Na_{2i}^N \right)=N^2\mathbb{D}^2 \left( \sum_{i=3}^N a_{1i}^Na_{2i}^N \right)= \\
& N^2\sum_{i=3}^N\mathbb{D}^2 \left(  a_{1i}^Na_{2i}^N \right)=N^2(N-2)\gamma^N(3,3,3,3)=O\left(N^3\left(p^N\right)^2\right)=O\left(N\langle d \rangle^2 \right).
\end{align*}

The second term in \eqref{eq:cvar} is bounded from above by 0 due to $\gamma^N(i,j,k,l)\geq 0$.

The third term in \eqref{eq:cvar} is
\begin{align}
\label{gamma}
&\mathbb{D}^2\left(d^N(1)d^N(2) \right)=\sum_{i=1}^N\sum_{j=1}^N\sum_{k=1}^N\sum_{l=1}^N\gamma^N(i,j,k,l).
\end{align}
\eqref{gamma} has $O(N^3)$ nonzero terms where either $i=k,j\neq l$ or $i\neq k,j=l$. Each of these terms is $O\left(\left( p^N\right)^3 \right)$, so their total contribution is $O\left(N^3(p^N)^3\right)=\left(\langle d \rangle^3  \right)=O \left(N\langle d \rangle^2 \right)$.
\eqref{gamma} also has $O(N^2)$ nonzero terms where $i=k$ and $j=l$. Each of these terms is $O\left( \left( p^N\right)^2\right)$, so their total contribution is $O\left(N^2(p^N)^2\right)=O\left(\langle d \rangle^2  \right)$.
 
Hence $\mathbb{D}^2\left(d^N(1)d^N(2) \right)=O\left(N\langle d \rangle^2  \right) $, so all three terms in \eqref{eq:cvar} are $O\left(N\langle d \rangle^2  \right) $, concluding the proof of the lemma.
\end{proof}

The next lemma is the counterpart of Lemma \ref{l:Dfismall}.

\begin{lemma}
\label{l:3} Assume A2. Then there exist $G^N$-measurable random variables $R_2^N$ such that
\begin{align}
& \sum_{s=1}^S \mathbb{D}^2_{G, \xi}\left( \phi_{\eta,s}^N(t) \right) \leq R_2^N.
\end{align}
and $R^N_2 \xrightarrow{p} 0$ as $N\to\infty$.
\end{lemma}

\begin{proof}
Just like for Lemma \ref{l:Dfismall}, we have
\begin{align}
\label{eq:R2general}
&\sum_{s=1}^S \mathbb{D}^2_{G, \xi}\left( \phi_{\eta,s}^N(t) \right) \leq \frac{1}{\langle d \rangle^2}\sum_{i=1}^N \sum_{j=1}^N |c^N(i,j)|=R_2^N.
\end{align}
 
Since $0 \leq c^N(i,i) \leq \frac{1}{N}\sum_{k=1}^N \left( a_{ik}^N\right)^2$, the diagonal terms in \eqref{eq:R2general} can be handled with
\begin{align*}
& \E \left [\frac{1}{\langle d \rangle^2}\sum_{i=1}^N \left |c^N(i,i) \right | \right]=\frac{N}{\langle d \rangle^2} \E \left |c^N(1,1) \right | \leq  \frac{N}{\langle d \rangle^2} \E \left [\frac{1}{N}\sum_{k=1}^N \left( a_{1k}^N\right)^2\right]= \\
&\frac{N-1}{\langle d \rangle^2}p^N\left(\sigma^2+\mu^2\right)=O\left( \frac{1}{\langle d \rangle}\right) \to 0.
\end{align*}

For the non-diagonal terms in \eqref{eq:R2general} we first calculate $\E\left(c^N(1,2) \right)$.
 
\begin{align*}
&\E\left(c^N(1,2) \right)=\frac{1}{N}\sum_{k=1}^N \E\left(a_{1k}^Na_{2k}^N \right)-\frac{1}{N^2}\sum_{k=1}^N
\sum_{l=1}^N\E\left(a_{1k}^Na_{2l}^N \right)= \\
& \frac{1}{N}\sum_{k=3}^N \E\left(a_{1k}^Na_{2k}^N \right)-\frac{1}{N^2}\sum_{k=3}^N
\sum_{l=3}^N\E\left(a_{1k}^Na_{2l}^N \right)-\frac{1}{N^2}\E \left(a_{12}^N \right)^2= \\
& \frac{N-2}{N}\left(p^N \mu \right)^2- \left(\frac{N-2}{N}\right)^2\left(p^N \mu \right)^2+O\left( \frac{\left( p^N\right)^2}{N^2}\right)= \\
& \left(p^N \mu \right)^2 \frac{N-2}{N}\left[1-\frac{N-2}{N} \right]+O\left( \frac{\left( p^N\right)^2}{N^2}\right)= 
%\\ &\left(p^N \mu \right)^2 \frac{N-2}{N^2}+O\left( \frac{\left( p^N\right)^2}{N^2}\right)=
O\left( \frac{\left( p^N\right)^2}{N}\right).
\end{align*}
Hence the non-diagonal terms in \eqref{eq:R2general} have the upper bound
\begin{align*}
&\frac{2}{\langle d \rangle^2}\sum_{i<j}  \left |c^N(i,j) \right | \leq \frac{2}{\langle d \rangle^2}\sum_{i<j}  \left |c^N(i,j)-\E \left(c^N(1,2) \right) \right |+\frac{2}{\langle d \rangle^2}\sum_{i<j}  \left | \E \left( c^N(1,2) \right) \right |= \\
&   \frac{2}{\langle d \rangle^2}\sum_{i<j}  \left |c^N(i,j)-\E \left(c^N(1,2) \right) \right |+O\left(\frac{1}{(Np^N)^2}\cdot N^2\cdot\frac{(p^N)^2}{N} \right) = \\
&   \frac{2}{\langle d \rangle^2}\sum_{i<j}  \left |c^N(i,j)-\E \left(c^N(1,2) \right) \right |+O\left(\frac{1}{N} \right).
\end{align*}
The expectation of the first term can be estimated using Lemma \ref{l:2}:
\begin{align*}
& \E \left [\frac{2}{\langle d \rangle^2}\sum_{i<j}  \left |c^N(i,j)-\E \left(c^N(1,2) \right) \right | \right]=\frac{N(N-1)}{\langle d \rangle^2}\E \left|c^N(1,2)-\E \left(c^N(1,2) \right) \right| \leq \\
&\frac{N^2}{\langle d \rangle ^2}\mathbb{D} \left(c^N(1,2) \right)=\frac{1}{\langle d \rangle ^2}\mathbb{D} \left(N^2c^N(1,2) \right)=O\left( \frac{\sqrt{N}}{\langle d \rangle} \right)=O\left( N^{-\epsilon} \right) \to 0,
\end{align*}
where the last equality is due to A2. Therefore, $ R_2^N \xrightarrow{p}0$.
\end{proof}

Lemma \ref{l:3} shows that $\phi_i^N(t) $ does not vary much for a typical vertex $i$, and Lemma \ref{l:1} shows that the average of all $\phi_i^N(t)$ is close to the global average $x^N(t)$. Our next aim is to show that replacing each local environment by $x^N(t)$ results in a small error. This is where the condition A1 or B2 will play a role. Without either A1 or B2, it might be the case that $Q_{sk}$ increases rapidly, incurring a large global error resulting from the few vertices where $\phi_{i}^N(t)\approx x^N(t)$ breaks down. Ensuring either that there are no such vertices at all (which will follow from B2) or that $Q_{sk}$ does not increase rapidly (A1) eliminates this issue.

The magnitude of $ \phi_i^N(t) $ is directly related to the degrees $d^N(i)$ due to
\begin{align*}
& \left \| \phi_{i}^N(t) \right \|=\sum_{s=1}^S \left | \phi_{i,s}^N(t) \right|=\sum_{s=1}^S \frac{1}{\langle d \rangle }\sum_{j=1}^N a_{ij}^N \xi_{i,s}^N(t)= \frac{1}{\langle d \rangle }\sum_{j=1}^N a_{ij}^N=\frac{d^N(i)}{\langle d \rangle }.
\end{align*}   

%If the $G^N$ is sparse than in the \erg{} case the maximum degree can be arbitrarily many times larger than the average degree which means we must restrict $Q$ globally with some regulatory conditions and A1 seems unavoidable. However, in the dense case when B1 is true we can guarantee that the arguments remains bounded and $Q$ only needs to be Lipschitz continuous in a compact set which follows from its local Lipschitz continuity. To achieve this we have to make sure that the weights do not have too erratic behaviours hence condition B2.

For some $M>1$ we define
$$\Delta_{M}:=\{ v \in \mathbb{R}^S : v_{s} \geq 0\,\forall s \in \mathcal{S}, \ \sum_{s=1}^S v_{s} \leq M \}$$
and the event
\begin{align*}
 & \mathcal{L}_M^N:=\bigcap_{i=1}^N\left \{ \frac{d^N(i)}{\langle d \rangle } \leq M \right \}.%=\bigcap_{i=1}^N \left \{ \phi_{i}^N(t) \in \Delta_{M}\right \}.
 \end{align*}
 
\begin{lemma}
\label{l:4}
Assume B1,B2. Then there exists some $1<M<\infty$ not depending on $N$ such that
$$\pr\left(\mathcal{L}^N_M \right) \to 1$$
as $N\to\infty$.
\end{lemma}    

\begin{proof} 
According to B2, $ \log \Lambda(s)=\log \left( \E \exp \left( w_{12}^N s\right) \right)$ is finite for some $s>0$. Using that $s$, we have
\begin{align*}
&\pr\left(\bar{\mathcal{L}}_M^N \right)=\pr \left(\bigcup_{i=1}^N \left \{ \frac{d^N(i)}{\langle d \rangle} >M \right \} \right) \leq \sum_{i=1}^N \pr \left(\frac{d^N(i)}{\langle d \rangle} >M\right)= \\
&N\pr \left(\frac{d^N(1)}{\langle d \rangle} >M\right)=N\pr \left( \frac{1}{\langle d \rangle}\sum_{i=2}^{N}w_{1i}^N >M \right)= \\
& N\pr \left( \frac{s}{N-1}\sum_{i=2}^{N}w_{1i}^N >sM\mu \right) \leq N e^{-sM\mu(N-1)}\left[E(e^{sw_{12}^N})\right]^{N-1}
\end{align*}
due to exponential Chebyshev's inequality. Then $M$ is chosen large enough so that the exponent on the right hand side is negative and $\pr\left(\bar{\mathcal{L}}_M^N \right)\to 0$ as $N\to\infty$.
\end{proof}

$M$ is fixed from now on. When $\mathcal{L}_M^N$ is true, $Q$ can be restricted to $\Delta_{M}$ which is compact; $L_{sk}$ and $L_Q$ denote the corresponding Lipschitz-constants of $Q$ on this domain.

$\Lleq$ will denote that inequality holds when $\mathcal{L}_M^N$ is true.

\begin{lemma}
\label{l:5}
Assume either (A1, A2) or (B1, B2). Then for any $T>0$ there exist $G^N$-measurable random variables $R^N$ such that
\begin{align}
& \sup_{0\leq t\leq T} \|H^N(t)\| \leq R^N
\end{align}
and $R^N \xrightarrow{p} 0 $ as $N\to \infty $.
\end{lemma}

\begin{proof}
Assumption A1 provides global Lipschitz continuity, so the proof of Lemma \ref{l:hnt2} remains valid in this case.
 
Assuming B1 and B2, we can still use \eqref{l5:useful:1} when $\mathcal{L}_M^N$ is true to get 
\begin{align*}
& \left \| \frac{1}{N}\sum_{i=1}^NQ \left(\phi_i^N(t) \right)\xi_{i}^N(t)-f \left( x^N(t) \right) \right \| \ \Lleq L_Q\left(R_1^N+\sqrt{SR_2^N} \right).
\end{align*}

On $\bar{\mathcal{L}}_M^N$ we use the trivial upper bound
\begin{align*}
& \left \| \frac{1}{N}\sum_{i=1}^NQ \left(\phi_i^N(t) \right)\xi_{i}^N(t)-f \left( x^N(t) \right) \right \|  \leq 2 \max_{v \in \Delta_{M'}} \left \|Q(v) \right \|
\end{align*}
where $ M':=\frac{\max_{1 \leq i \leq N}d^N(i)}{\langle d \rangle} $. Then
\begin{align*}
&R^N:= 
\begin{cases}
 L_Q\left(R_1^N+\sqrt{SR_2^N} \right) \ \textit{if} \ \mathcal{L}^N \\
 2 \max_{v \in \Delta_{M'}} \left \|Q(v) \right \|  \ \textit{if} \ \bar{\mathcal{L}}^N.
\end{cases}
\end{align*}
Clearly, $\|H^N(t)\| \leq R^N$ for all $t\geq 0$. It is also true that $R^N \xrightarrow{p}0$ since
\begin{align*}
&\pr\left(R^N > \varepsilon \right) \leq  \pr \left( \bar{\mathcal{L}}^N \right)+\pr\left( L_Q\left(R_1^N+\sqrt{SR_2^N} \right)>\varepsilon \right) \to 0
\end{align*}
as $N\to\infty$.
\end{proof}

\begin{lemma}
\label{l:6} 
Assume either (A1, A2) or (B1, B2). Then for any $T>0$,
\begin{align*}
& \sup_{0 \leq t \leq T}\left \| K^N(t) \right \| \xrightarrow{p} 0.
\end{align*}
\end{lemma}	

\begin{proof} 
It is sufficient to show that for any $s \neq k$ in $\mathcal{S}$ and for all $T >0$, we have
\begin{align*}
\sup_{0 \leq t \leq T} \left |K_{sk}^N(t) \right | \xrightarrow{p} 0.
\end{align*}

$s$ and $k$ are fixed for the rest of the proof of Lemma \ref{l:6}.

First we assume (B1, B2). Then, using Lemma \ref{l:4},
\begin{align*}
& \int_{0}^t \frac{1}{N}\sum_{j=1}^N Q_{sk}\left( \phi_{j}^N(\tau) \right)\xi_{k,j}^N(\tau) \d \tau \Lleq \int_{0}^{t} Q_{sk}(0)+L_{sk} \frac{1}{N}\sum_{i=1}^N \left \|\phi_{i}^N(\tau) \right \| \d \tau \Lleq \\
&\left(Q_{sk}(0)+ML_{sk} \right)T=:C,
\end{align*}
so we can use the functional strong law of large numbers for the Poisson process:
\begin{align*}
& \pr \left( \sup_{0 \leq t \leq T} \left |K_{sk}^N(t) \right | > \varepsilon \right) \leq \pr \left( \bar{\mathcal{L}}^N\right)+\pr \left( \sup_{0 \leq t \leq C} \left | \frac{1}{N} \mathcal{N}_{sk}(Nt)-t \right | > \varepsilon \right) \to 0
\end{align*} 
as $N\to\infty$.

Assume now (A1,A2). For some $W>0$, let $A_W$ denote the event
$$A_W:=\left\{\frac{1}{N\langle d \rangle}\sum_{i=1}^{N} d^N(i)>W \right \}.$$

From Markov's inequality, it is clear that $\pr(A_W) \leq \frac{1}{W}$, and
\begin{align*}
& \int_{0}^t \frac{1}{N}\sum_{j=1}^N Q_{sk}\left( \phi_{j}^N(\tau) \right)\xi_{k,j}^N(\tau) \d \tau \leq \int_{0}^{t} Q_{sk}(0)+L_{sk} \frac{1}{N}\sum_{i=1}^N \left \|\phi_{i}^N(\tau) \right \| \d \tau \overset{\bar{A}_W}{\leq} \\
&\left(Q_{sk}(0)+WL_{sk} \right)T=:C.
\end{align*}
Once again, we use the functional strong law of large numbers for the Poisson process:
\begin{align*}
& \limsup_{N \to \infty} \pr \left( \sup_{0 \leq t \leq T} \left |K_{sk}^N(t) \right | > \varepsilon \right) \leq \\
& \limsup_{N \to \infty }\pr \left( A_W\right)+\limsup_{N\to\infty}\pr \left( \sup_{0 \leq t \leq C} \left | \frac{1}{N} \mathcal{N}_{sk}(Nt)-t \right | > \varepsilon \right) \leq \frac{1}{W}.
\end{align*} 
$W$ can be chosen arbitrarily large, ensuring
$$\limsup_{N \to \infty} \pr \left( \sup_{0 \leq t \leq T} \left |K_{sk}^N(t) \right | > \varepsilon \right)= 0.$$
\end{proof}

With Lemmas \ref{l:5} and \ref{l:6}, the bounds for $K^N(t)$ and $H^N(t)$ are in place, and the proof of Theorem \ref{t:1} concludes identically to the proof of Theorem \ref{t:simple} using Gr\"onwall. $\Box$

\section{Conclusion}
\label{s:concl}

The paper provides rigorous proof for the mean-field convergence of local density-dependent Markov processes on weighted {\erg}s to the homogeneous mean-field limit in the transient regime under relatively mild density and regularity conditions. The proof is also carried out for the deterministic NIMFA process.

One natural question is how far the exponent $(1/2+\varepsilon)$ in condition A2 can be decreased to still have the same mean-field limit. This is subject to further investigation.

Other future work possibly includes other models of random graphs. In general, we conjecture that a similar mean-field limit holds for other graph models too, as long as the underlying graph structure is sufficiently dense. However, for sparse graphs, the homogeneous mean-field approximation is known to fail in certain scenarios. We expect that for sparse graphs, the mean-field limit may depend heavily on the underlying graph structure and possibly also on details of the process. This is also subject to further investigations.
\section{Appendix}
\label{s:app}

In this section we prove some properties of NIMFA.

\begin{lemma}
	\label{l:appendix}
With initial condition $z_i^N(0) \in \Delta$, the solution of \eqref{eq:deterministic} exists and is unique for all $t\geq 0$ and remains in $\Delta$.
\end{lemma}
\begin{proof}
Let $N$ be fixed. Due to the right hand side of \eqref{eq:deterministic} being locally Lipschitz-continuous, we have local existence and uniqueness. Next we prove $z_i^N(t) \in \Delta$ for any $t>0$.

We create $N$ auxiliary processes $\zeta_{i}^N(t)$, which are independent time in\-ho\-mo\-ge\-neous Markov processes with time-dependent transition rate matrices $Q(\rho_i^N(t))$. We use the notation $$v_{i,s}^N(t):=\pr\left(\zeta_{i,s}^N(t)=1\right).$$

The Chapman-Kolmogorov equations for $v_{i,s}^N(t)$ are
\begin{align*}
& \frac{\d}{\d t} v_i^N(t)=Q(\rho_i^N(t))v_{i}^N(t).
\end{align*}

We claim that if $z_i^N(0)=v_i^N(0)$ for all $1\leq i\leq N$, then we also have $z_i^N(t)=v_i^N(t)$ for any $t>0$.
\begin{align*}
& \sum_{i=1}^N \|v_i^N(t)-z_i^N(t) \| \leq  \\
&\sum_{i=1}^N \|v_i^N(0)-z_i^N(0) \|+ \int_{0}^{t} \|Q(\rho_i^N(\tau))[v_i^N(\tau)-z_i^N(\tau)]  \| \d \tau \leq \\
&\sum_{i=1}^N \|v_i^N(0)-z_i^N(0) \|+ \underbrace{ \max_{i}\max_{ 0 \leq \tau \leq t}\|Q(\rho_{i}^N(\tau)) \|}_{C_t}\int_{0}^{t} \|Q(\rho_i^N(\tau))v_i^N(\tau)-z_i^N(\tau)  \| \d \tau
\end{align*}
Gr\"onwall's lemma then implies
\begin{align*}
& \sup_{0 \leq \tau \leq t}\sum_{i=1}^N \|v_i^N(\tau)-z_i^N(\tau) \| \leq \sum_{i=1}^N \|v_i^N(0)-z_i^N(0) \|e^{C_t t}=0. 
\end{align*}

Obviously $v_i^N(t)=(v_{i,s}^N(t))_{s\in\mathcal{S}}\in\Delta$ since it is a probability vector, so $z_i^N(t) \in \Delta$ also holds. But $\Delta$ is a compact set, so $Q$ is globally Lipschitz-continuous on $\Delta$, and global existence and uniqueness follows.

\end{proof}

\bibliographystyle{abbrv}
\bibliography{mf}

\end{document}